\documentclass[a4paper,11pt]{amsart}

\usepackage{epsfig}
\usepackage[usenames, dvipsnames]{color}
\usepackage{amssymb}

\newtheorem{theorem}{Theorem}
\newtheorem{corollary}[theorem]{Corollary}

\newtheorem{lemma}[theorem]{Lemma}

\newcommand{\RR}{{\mathbf R}}
\newcommand{\EU}{{\mathbf S}}
\newcommand{\vv}{{\mathbf v}_{+}}
\newcommand{\ww}{{\mathbf v}_-}
\newcommand{\vpm}{{\mathbf v}_\pm}

\newcommand{\ee}{{\mathrm e}}
\newcommand{\Rc}{{\mathcal R}}

\newcommand{\dpt}{\displaystyle}

\title[High frequency forcing of an attracting heteroclinic cycle]{High frequency forcing \\ of an attracting heteroclinic cycle \\}

\author[I. S. Labouriau and A. A. Rodrigues]{Isabel S. Labouriau$^{1}$ and Alexandre A. Rodrigues$^{1,2}$ \\
 \\ 
$^{1}$Centro de Matem\'atica da Universidade do Porto, \\ \medskip Rua do Campo Alegre s/n, Porto 4169-007, Portugal \\
$^2$Lisbon School of Economics \& Management, \\Rua do Quelhas 6,  Lisboa 1200-781, Portugal 
}

\thanks{The first author was partially supported by CMUP (UID/MAT/00144/2020), which is funded by FCT with national (MCTES) and European structural funds through the programs FEDER, under the partnership agreement PT2020.     The second author was supported by the Project CEMAPRE/REM -- UIDB /05069/2020 financed by FCT/MCTES through national funds. }
\email{ islabour@fc.up.pt \quad alexandre.rodrigues@fc.up.pt, arodrigues@iseg.ulisboa.pt }
\begin{document}

\begin{abstract}
This article is concerned with the effect of time-periodic forcing on a vector field exhibiting an attracting heteroclinic network.
We show that as the forcing frequency tends to infinity, the dynamics reduces to that of a network under constant forcing, the constant being the average value of the forcing term.
We also show that under small constant forcing the network breaks up into an attracting periodic solution that persists for periodic forcing of high frequency.

\end{abstract}

    \maketitle
  
  \bigbreak

\bigbreak
\textbf{2010 --- AMS Subject Classifications} 
{Primary: 37C60;   Secondary: 34C37, 34D20, 37C27, 39A23}

\bigbreak

\section{Introduction}
Heteroclinic cycles organize the dynamics in a wide range of systems: ecological models of competing species \cite{AHL2001, GH88,  ML75}, thermal convection \cite{PJ88, Rodrigues2013},  game theory {\cite{CFGL,GSC20}} and climate science \cite{APW}.   
A paradigmatic example of a robust heteroclinic cycle occurs Guckenheimer and Holmes  three-dimensional system \cite{GH88},  also studied by May and Leonard \cite{ML75} and by Busse and Heikes \cite{BH80}.  Although their initial models have periodic forcing terms, all the theory has been developed for the autonomous case. In this case, the equations are symmetric under permutation of coordinates.  

With the non-autonomous forcing terms removed,  each of the equilibria on the coordinate axes is of saddle type, and the existence of connecting orbits has been proved. 
Moreover, attracting heteroclinic   networks have been found in an open set in the   space of parameters \cite{GH88}. 
 Other examples from the dissipative category include the equations of Lorenz, Duffing
 and Lorentz gases acted on by external forces \cite{CELS}. 

Symmetry-breaking constant perturbations to robust heteroclinic cycles with one-dimensional connections are well-known  to result in long-period periodic solutions that lie close to the original cycle.  To date there has been very little systematic investigation of the effects of perturbations that are time-periodic, despite  being natural for the modelling of many biological effects  as for instance the effect of seasonality in epidemiological models  \cite{CR22}.

  Mathematically, one might expect to make comparisons between the consequences
  of time-periodic forcing on a heteroclinic cycle and the well-known effects of time-periodic forcing on periodic oscillations, for example frequency-locking (existence of periodic solutions whose period is  an integer multiple of the period of the external forcing) and the effects of low/high frequency. These general observations provide double
   motivations for this work.
 
  In a series of papers, the authors of \cite{DT3, Rabinovich06,   TD2, TD1} considered the effect of small-amplitude time-periodic forcing of an attracting heteroclinic network and describe how to reduce the dynamics to a two-dimensional map. 
 In the limit where the heteroclinic cycle is weakly attracting, intervals of frequency locking appear.  In the opposite limit, where the heteroclinic cycle becomes strongly stable,  no frequency locking is observed. In   \cite{Rodrigues2021} it is proved that  strange attractors are abundant near a heteroclinic cycle when the frequency  $\omega$ of the forcing term satisfies $\omega \approx 0$, emphasizing the   important role of $\omega$ on the dynamics.   This discussion raised the question: \\
 \begin{description}
 \item[(Q)] could we identify and isolate the asymptotic  effect of $\omega$ on the dynamics near a periodically-perturbed robust heteroclinic cycle?\\
 \end{description}
 There is no general theory for high-frequency forcing on attracting heteroclinic networks although  the analysis of its effect   in nonlinear oscillations may be useful in many applications. For instance, a high-frequency signal   may
 block the conduction in the neural system  \cite{WK98}.

In the present article, we examine the effect of
  high frequency
 periodic forcing on a nonlinear system of differential equations constructed in
  Aguiar \emph{et al}  \cite{ACL06}
 which, in the absence of forcing, exhibits
 an asymptotically stable  network.  
We prove that the dynamics of the system is equivalent to the averaged one as the frequency of the forcing tends to $+\infty$. Although our computations are 
 done on the example of \cite{ACL06},
our  results are of far wider interest than the specific problem studied in this paper.

 \subsection*{Structure}
 This article is organized as follows.
 In Section \ref{s: object} we describe our object of study and,  for the sake of completeness, we   review
 in Section \ref{gamma=0}  some properties of the unperturbed system. 
 We also motivate the model  and describe how the the problem under consideration fits in the literature. 
 
 The remainder of the article is dedicated to the proof of the main result. 
In Section \ref{cross_sections}, we construct suitable cross sections near the equilibria where the local and global maps will be defined. 
 In Section \ref{secFirstReturn}, we  obtain the expressions that will be used to compute the first return map to a given   cross-section. We also derive   auxiliary results that will be helpful to analyse the asymptotic coefficients of the  first return map.
Section \ref{sec:mainResult} proves the main contribution of this article
 as well as some  dynamical consequences. 
 The existence of an attracting and hyperbolic solution near the cycle  under high frequency forcing is explored in Section~\ref{sec:DinOmegaLarge}.
 We finish the article with a short discussion in Section \ref{ss: Discussion}. In Appendix \ref{s: Notation}, we list some notation used throughout this article, as well their meaning. 
\section{The model under consideration}
\label{s: object}

Our
object of study is the  following two-parameter family of ordinary differential equations $$\dot X=F_{(\nu,\mu)}(X,t)$$ defined in $X=(x,y,z)\in \RR^3$ by:
 \begin{equation}
\label{general}
\left\{ 
\begin{array}{l}
\dot x = x(1-r^2)-\alpha x z +\beta xz^2 + (1-x) [\nu f(2\omega t)+\mu ]\\
\dot y =  y(1-r^2) + \alpha y z + \beta y z^2 \\
\dot z = z(1-r^2)-\alpha(y^2-x^2)-\beta z (x^2+y^2) 
\end{array}
\right.
\end{equation}
where 
\begin{itemize}
\item $r^2=x^2+y^2+z^2$,
\item  $ \omega \in \RR^+$,
\item $\mu ,\nu \in \RR^+_0$ are two small and independent  
 parameters and 
\item $f:\RR \to \RR$ is a smooth, non-constant $2\pi/\omega$--periodic map such that $$
 \int_0^{\pi/\omega}
 f(2\omega t)dt=0.
$$
 \end{itemize}
\bigbreak
We also assume that: 
 \begin{equation}
\label{parameters}
\beta<0<\alpha, \qquad 
  |\beta|<\alpha
 \quad\Rightarrow\quad
 \beta^2<8 \alpha^2  .
\end{equation}
\bigbreak

Concerning the equation \eqref{general}, the amplitude of the autonomous perturbation is governed by the parameter $\mu$ whereas $\nu$ controls the amplitude of the non-autonomous term. The parameter $\omega$ is what we call the \emph{frequency} of the periodic forcing. We refer to   $\dot X=F_{(0,0)}(X,t)$ as the \emph{unperturbed system}.

\section{Motivation and state of the art}
\label{gamma=0}
In this section we recall some basic features associated to the system \eqref{general} when $\nu=\mu=0$. For the terminology of equivariant differential equations and heteroclinic structures we refer the reader to the book by Golubitsky and Stewart \cite{GS}.
\subsection{The construction of the unperturbed system}
The equation \eqref{general}
was originally  constructed to obtain
 a symmetric   heteroclinic network associated to two equilibria, using a general construction described in  \cite{ACL06}, that we proceed to summarise.  \\

Start with the differential equation
$$\dot{X}= (1- \|X\|^2)X$$
   for $X = (x,y,z)\in \RR^3$, 
  for which the unit sphere $\EU^2$ attracts all points except the origin (\emph{ie} is globally attracting) and all its points are equilibria.
  Then consider  the
  finite  Lie group $ {\mathcal G}\subset \mathbb{O}(3)$  generated by
  the two linear maps:
$$
\kappa_1(x,y,z)=(-y,x,-z)
\qquad\mbox{and}\qquad
\kappa_2(x,y,z)=(x,-y,z) .
$$

Add two $ {\mathcal G}$--equivariant perturbing terms of order 3, say $\alpha A(X)$ and $\beta B(X)$, to the
differential equation. These two terms are chosen to be tangent to $\EU^2$, so this sphere is still flow-invariant and attracting.
 The new equations are then $ {\mathcal G}$-equivariant, and therefore, for each subgroup ${\mathcal H}\subset{\mathcal G}$ the fixed point  subspace $Fix({\mathcal H})=\{X\in \RR^3:\kappa(X) = X\ \forall \kappa \in{\mathcal H} \} $ is flow-invariant (see \cite{GS}).
Hence the 
  action of $ {\mathcal G}$ on $\RR^3$ has the following symmetry flow-invariant planes   corresponding to subgroups generated by  elements $\kappa\in {\mathcal G}$ such that $\kappa^2$ is the identity: \\
\begin{eqnarray*}
Fix(\kappa_2 \circ \kappa_1^2) &=\{{ X}
 \in \RR^3:\kappa_2 \circ \kappa_1^2(X) = X \} = \{(x, y, z)
\in \RR^3:  x=0\}, \\
Fix(\kappa_2  ) & =\{{ X}
\in \RR^3:\kappa_2 \circ \kappa_1^2(X) = X \} = \{(x, y, z)\in \RR^3:  y=0\},
\end{eqnarray*}\\
and symmetry axes:\\
\begin{eqnarray*}
Fix(\kappa_1^2) &=& \{(x, y, z)\in \RR^3:  x=0\, \quad  \text{and} \,  \quad y=0\}, \\ 
Fix(\kappa_2\circ \kappa_1^3) &=& \{(x, y, z)\in \RR^3:  x=y\, \quad \text{and} \, \quad  z=0\}, \\
 Fix(\kappa_2\circ \kappa_1) &=&\{(x, y, z)\in \RR^3:  x=-y\, \quad \text{and} \,  \, z=0\}.\\
\end{eqnarray*}

\subsection{Dynamics of the unperturbed system}
\label{ss:unp}
 
The intersection of the flow-invariant sphere $\EU^2$ and 
 $Fix(\left\langle \kappa \right\rangle  )$, $\kappa\in \mathcal{G}$, 
 is a flow-invariant set. Then, as illustrated in Figure~\ref{fig: sections}, the intersection of this sphere with $Fix(\kappa_1^2)$  gives rise to two saddle-type equilibria
$$ \vv =(0, 0,1), \qquad \text{and} \qquad  \ww =(0,0,- 1),$$
where the derivative of $F_{(0,0)}$ is
$$
DF_{(0,0)}(0,0,\sigma)=\begin{pmatrix}
\beta-\sigma\alpha&0&0\\
0&\beta+\sigma\alpha&0\\
0&0&-2
\end{pmatrix}
\qquad \text{where}\qquad \sigma=\pm 1.
$$

On $\EU^2$  there are also four  unstable foci $\left(\pm \frac{\sqrt{2}}{2}, \pm \frac{\sqrt{2}}{2}, 0\right) $
  on the lines 
  $Fix(\kappa_2\circ \kappa_1^3)$ and $Fix(\kappa_2\circ \kappa_1)$. 
The intersections of the sphere $\EU^2$ with  the planes $Fix(\kappa_2 \circ \kappa_1^2)$ and 
$Fix(\kappa_2  )$ generate two pairs of one-dimensional heteroclinic connections linking the equilibria $\vv$ and $\ww$, as depicted in Figure \ref{fig: sections}.   

 The union of these equilibria and connections forms four heteroclinic cycles and, taken together, a heteroclinic network denoted $\Gamma$.
 The two-dimensional coordinate subspaces $Fix(\kappa_2 \circ \kappa_1^2){ =\{x=0\} }$ and $Fix(\kappa_2  ){=\{y=0\} }$ are flow-invariant and separate the space $\RR^3$, hence  trajectories starting on a connected component of $\RR^3\backslash \left(\{x=0\} \cup \{y=0\} \right)$ cannot visit another component.
 Thus, nearby trajectories only visit a neighbourhood of one
cycle in the network,  hence there is no \emph{switching} in the sense of  \cite{ACL NONLINEARITY}.  

Within each of the invariant planes defined by $x = 0$   and ${y = 0}$, the 
connecting orbit is a saddle-sink connection.
 Therefore the network $\Gamma$ is robust in the class of ${\mathcal G}$-symmetric vector fields.\\

 Define the map $g:\RR^3 \to \RR^3$ such that $g(x,y,z)=(x-y)^2+z^2$. If $\beta=0$, then the Lie derivative of $F_{(0,0)}$ with respect to $g$ is identically zero  on $\EU^2$, which means that  trajectories of the flow of \eqref{general} on $\EU^2$ move along closed trajectories that are the intersection of $\EU^2$ and the level surfaces of $g$. Therefore, the perturbation $\beta B(X)$ and Condition \eqref{parameters} force the foci to be \emph{unstable} when restricted to $\EU^2$, and the network $\Gamma$ to be    \emph{globally asymptotically stable} (cf. \cite{ACL06}).  Typical trajectories starting near $\Gamma$   accumulate on   one of the cycles in the network and  remain near the equilibria for increasing periods of time. These trajectories make fast transitions from one equilibrium point to the next. In particular, there are no periodic solutions near $\Gamma$. \\

The constant   $ \delta=  \frac{\alpha-\beta}{\alpha+\beta}>1$ 
measures the \emph{strength of attraction} of each equilibrium because it is the ratio between the contracting
(negative  tangent to $\EU^2$)
and the expanding (positive) eigenvalues  of $DF_{(0,0)}(0,0,\sigma)$. It is related with the ratio of consecutive times of sojourn near each equilibrium. Analogously the constant $\delta^2$ measures the \emph{strength of attraction} of the each cycle of $\Gamma$.

  \begin{figure}
\begin{center}
\includegraphics[width=11cm]{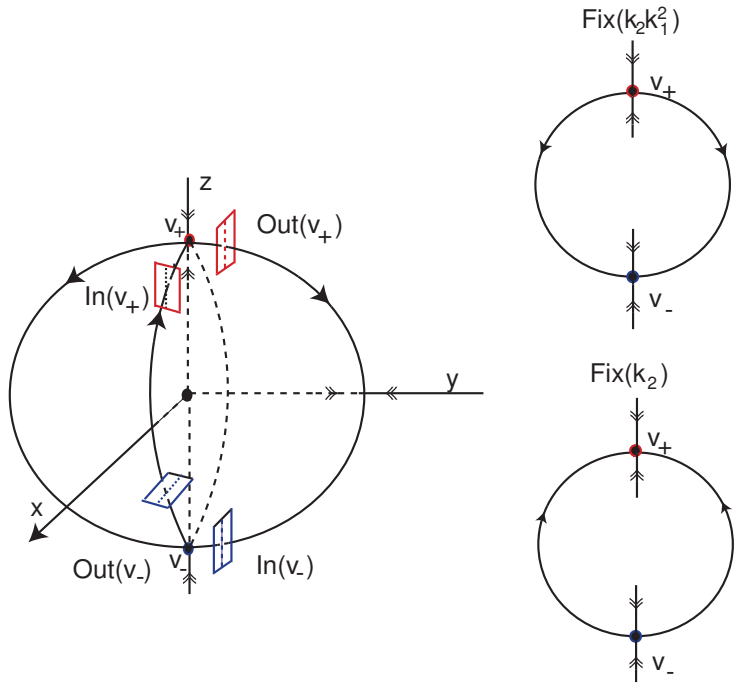}
 
\end{center}
\caption{\small Sketch of the heteroclinic connections when  $\mu=\nu=0$. }
\label{fig: sections}
\end{figure}
  \subsection{The perturbing terms}
  
 We add  to \eqref{general} two perturbing terms of different nature, one governed by $\mu$ and the other by $\nu$.  
 The term governed by $\mu$ is an autonomous perturbation; the other is non-autonomous. If  either $\nu>0$ or $\mu>0$, only the symmetry $\kappa_2$ remains and, in general, the network $\Gamma$ is destroyed. \\
  
  Our choice of perturbing term 
$ (1-x) [\nu f( 2\omega t)+\mu]$
 is made for two reasons: first, it simplifies the computations and allows comparison with previous work by other authors \cite{ACL06, DT3, LR18, LR21, Rabinovich06, TD1}. Secondly, it simplifies the quantitative reduction of the differential equations to a map on a cylinder. The dynamics of the unperturbed equations will transfer the effect of the perturbation to the other coordinates.\\

We may split the family of perturbations as follows:
 $$
  (1-x) [\nu f( 2\omega t)+\mu]= \nu \underset{\text{non-autonomous forcing}}{  \underline{ (1-x) f( 2\omega t)}} + \mu\underset{\text{autonomous}}{\underline{ (1-x)}}
 $$
where:
 $$
 \int_0^{\pi/\omega}
 f(2\omega t)dt=0.
 $$
 This splitting will be clearer in the sequel.\\

 \subsection{State of the art}
When either $\mu\ne 0$ or $\nu\ne 0$ the plane defined by $x=0$ is no longer invariant due to the $(1-x)$ factor in the perturbing term, hence the connection from $\vv$ to $\ww$ is broken.
 The autonomous case $\mu>\nu=0$ is the simplest one and may be explored as in \cite[\S 5.2.1]{DT3} to conclude that the dynamics is governed by   an attracting periodic solution. 

The case $\nu>\mu=0$ and $\delta^2 \gtrsim 1 (\Leftrightarrow \delta \gtrsim 1)$ has been studied in \cite{LR18}, where the authors derived the first return map near a heteroclinic cycle for small amplitude of the perturbing term and reduce the analysis of the non-autonomous system to that of a two-dimensional map on a cylinder. 

They found rich dynamical features arising from a discrete-time Bogdanov-Takens bifurcation. When the perturbation strength is small, the first return map has an attracting invariant closed curve that is not contractible on the cylinder.  Furthermore, the authors   pointed out the existence of two distinct dynamical regimes corresponding to the existence or non-existence of intervals of frequency locking as $\omega$ varies. 
These findings are consistent with those of  \cite{TD2, TD1} who analysed a similar phenomenon near the May-Leonard cycle.

Increasing the perturbation strength there are periodic solutions that bifurcate into a closed contractible invariant curve and into a region where the dynamics is conjugate to a full shift on two symbols ($\Rightarrow$ chaos and strange attractors emerge).  

For the case $\nu>\mu=0$ and $\delta^2 \gg 1$,  the bistability described for the   case  $\delta^2 \gtrsim 1$ disappears. Authors of \cite{TD2, TD1} have shown the equivalence of the dynamics to that of a circle map, and discuss whether the circle map is likely to be invertible or non-invertible. The existence of regular and chaotic dynamics largely depends
on the order of $\omega$ as one may check in Table 3 of  \cite{Rodrigues2021}.

\subsection{Novelty}
 In the present paper, the forcing frequency is our principal bifurcation parameter and we are interested in the case $\omega$ large. 
 Pursuing an answer to question  \textbf{(Q)} in the context of strongly attracting systems ($\delta^2\gg 1$), this work  may be seen as the natural continuation of   \cite{LR18}.  
 
 Our main result states  that  when the   forcing
  frequency  $\omega \rightarrow +\infty$, the effect of the perturbation governed by $(\nu, \mu)$ is reduced to the effect of the autonomous term. In other words,  we show that  the asymptotic dynamics associated to $\dpt \lim_{\omega \rightarrow +\infty }F_{(\nu,\mu)}$ is qualitatively the same as that of the averaged system $F_{( 0,\mu)}$.  The proof   is performed via the construction of a first return map to a suitable cross section.

 As a consequence, when $\omega \to +\infty$, the non-autonomous equation (\ref{general}) has become autonomous, and the system dimension decreases as the non-autonomous component for ``fast motion'' governed by $\nu$ has been dropped completely.
    
  
\subsection{Insight into the reasoning}
In order to improve the readability of the article, we sketch the route of the proof of the main result of the paper.\\
 \begin{enumerate}
 \item
 We transform \eqref{general} into an autonomous equation in $\RR^4$ by adding a new coordinate $s$ for the forcing time. We call the new equation the {\em suspension of \eqref{general}}.  
\item We obtain  isolating blocks and cross sections near each equilibrium for the unperturbed differential equation. After the addition of non-autonomous terms to \eqref{general} the equilibria may no longer be equilibria, but the cross sections remain transverse to the flow  of the suspension;
\item   We derive the expressions for the linearization of \eqref{general}  near $\vv$ and $\ww$, say $\Phi_{\vv}$ and $\Phi_{\ww}$. The computation of these maps is  a bit more  tricky than
in the autonomous hyperbolic case because some  of the times of flight depend on $f$. These maps depend on the phase space coordinates $(x_2,w_2)$ also on the   suspension  time $s$; 
\item The local maps $\Phi_{\vv}$ and $\Phi_{\ww}$
 depend on  two expressions $K_1(\nu,\mu)$ and $K_2(\nu,\mu)$, where $$K_1(0,0)~=~ K_2(0,0)~=0.$$  
\item In Subsection  \ref{ss:aux}, we assume that $f(t)=\sin t$ in order to make  analytic progresses in the explicit computation of $K_1(\nu,\mu)$ and $K_2(\nu,\mu)$ although we emphasise that the result is valid for any $f$ satisfying the conditions of Section  \ref{s: object};
 
 \item After defining suitable global maps (natural in our scenario), we define the first return map $\Rc_{(\nu,\mu)}$ as the composition of local and global maps, and we prove the main result of this paper, the content of  Theorem \ref{prop:main}: $$\dpt \lim_{\omega\to\infty}\Rc_{(\nu,\mu)}=\Rc_{(0,\mu)}.$$ 
 
\item Finally, for the sake of completeness, we prove that the dynamics of $\Rc_{(0,\mu)}$ is governed by a sink  whose period goes to $+\infty$ as $\mu $ goes to $0$. When $\mu=0$, the sink collapses into
$\Gamma$. 
\end{enumerate}

 \section{Cross-sections}
 \label{cross_sections}
 
 Our results will be obtained analysing the first return  map to a suitable cross-section to the flow of \eqref{general}, 
 obtained from  the transitions between four cross-sections for the unperturbed equation. In this section, we construct suitable cross sections near the equilibria where the local and global maps will be defined. To do this,
consider cubic neighbourhoods $V_{\pm}$ 
in $\RR^3$  of 
$\vpm$:
 $$
V_\sigma=\{(x, y, w), |x|<\varepsilon,  |y|<\varepsilon, |w|<\varepsilon\}
\qquad w=z-\sigma\qquad\sigma=\pm 1
$$
for $\varepsilon>0$  small. 
As suggested by Figure \ref{fig:neigh1}, we use the following cross-sections contained in the boundary of $V_+$. 
\begin{itemize}
\item 
 $\dpt In(\vv) =\{(\varepsilon, y, w), |y|<\varepsilon, |w|<\varepsilon\}$ with coordinates $ (y_1, w_1)$
It consists of points whose trajectories  go into $V_+$ in small positive time.

\item   
$\dpt Out(\vv) =\{(x,\varepsilon, w), |x|<\varepsilon, |w|<\varepsilon\}$  with coordinates $ (\hat{x}_1, \hat{w}_1)$
It 
consists of points whose trajectories  go out of $V_+$ in small positive time.
\end{itemize}
The cross-sections contained in the boundary of $V_-$ are:
\begin{itemize}
\item   
$\dpt In(\ww) =\{(x,\varepsilon, w), |x|<\varepsilon,  |w|<\varepsilon\}$ with coordinates $  (x_2, w_2) $
with points  whose trajectories go into $V_-$ in small positive time,

\item  
$\dpt Out(\ww) =\{(\varepsilon, y, w), |y|<\varepsilon,  |w|<\varepsilon\}$ with coordinates $(\hat{y}_2, \hat{w}_2) $
containing points  whose trajectories go out of $V_-$ in small positive time.
\end{itemize}

\begin{figure}
\begin{center}
\includegraphics[width=12cm]{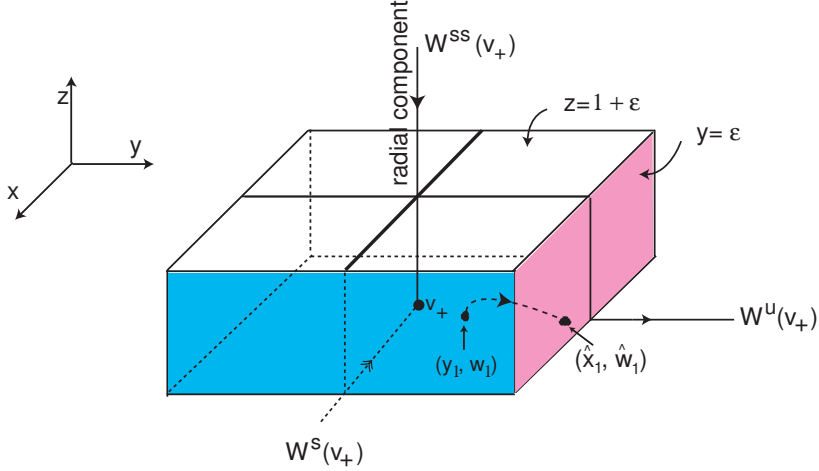}
 \end{center}
\caption{\small Scheme of the cross-sections $In(\vv)$ and $Out(\vv)$ in the neighbourhood $V_+$ of $\vv$. The set $W^{ss}(\vv)$, 
the local strong stable manifold of $\vv$,
corresponds to the radial direction in the invariant  sphere $\EU^2$.   }
\label{fig:neigh1}
\end{figure}

The local stable manifolds of $\vv$ and $\ww$ in the cross sections $In(\vv)$ and $In(\ww)$ are given by: 
$$
W^s(\vv)\cap In(\vv)=\{(0, w_1): |w_1|<\varepsilon\}
\qquad
W^s(\ww)\cap In(\ww)=\{(0, w_2): |w_2|<\varepsilon\}. 
$$

From now on we restrict our attention to the $y>0$ component of $In(\vv)\backslash W^s(\vv)$ (respectively the $x>0$ component of $In(\ww)\backslash W^s(\ww)$) and we abuse notation by calling it $In(\vv)$ (respectively $In(\ww)$).
All the results hold on the other  component,  but follow a different cycle in the heteroclinic network $\Gamma$.
  By rescaling the variables $(x,y,z)$ and the parameters $\mu$ and $\nu$ we may take $\varepsilon=1$ in the cross-sections defined above.

 \section{ Local maps}
\label{secFirstReturn}
The aim of this section is to  obtain expressions that will be used to compute the first return map to the cross-section $In(\vv)$ in the flow of \eqref{general}.
 The computations are similar to those of \cite{LR18}, so we give only an overview. 
The expression for the first return map to this  cross-section to $\Gamma$ is obtained  as the composition of two types of maps: \emph{local maps}  between the neighbourhood walls of each $V_\pm$,
and \emph{global maps} from one neighbourhood to the other.
 Here
we obtain the local maps (at leading order in $\nu,\mu$)
by computing
the point where a solution hits each cross-section and the  time  the solution takes to move between cross-sections. 
 
\subsection*{Suspension}
 For $(\nu,\mu)\ne(0,0)$ we consider the suspension of \eqref{general} given by
 \begin{equation}
 \label{eq:suspended}
\left\{\begin{array}{l} 
\dot X=F_{(\nu,\mu)}(X,s)\\
\dot s=1
\end{array}\right.
\qquad \mbox{where}\qquad
X=(x,y,z)\quad\mbox{and}\quad s\in\EU^1 
 \end{equation}
 and the augmented cross-sections $\EU^1 \times In(\vpm)$ and $\EU^1 \times Out(\vpm)$ to the suspended flow.
 Although $\vv$ and $\ww$ are no longer constant solutions of \eqref{general} for $(\nu,\mu)\ne(0,0)$,
 for small $(\nu,\mu)$  the augmented cross-sections defined above are still crossed transversely by trajectories.

\subsection{Linearisation}\label{secLinAut}
 The linearisation of \eqref{general} near $\vpm$ is
\begin{equation}
\label{linV}
\left\{ 
\begin{array}{l}
\dot x = (\beta-\sigma\alpha) x - \nu f(2\omega t) -\mu \\
\dot y =  (\alpha+\sigma\beta) y\\
\dot w = -2(w+\sigma)
\end{array}
\right.
\qquad
w=z-\sigma
\qquad
\sigma=\pm 1
\end{equation}
%
Equation \eqref{general} may be written in the form 
$$\
\dot X=\mathcal{M} X+R(X)- [\nu f(2\omega t) +\mu](1, 0, 0)^T
$$ 
for $X^T=(x,y,z)$  where
$\dot X= \mathcal{M} X- [\nu f(2\omega t) +\mu](1, 0, 0)^T$
 is any of the equations  \eqref{linV} for $\sigma=\pm 1$ 
 and  
 $$
 \mathcal{M}=
 \begin{pmatrix}
\beta-\sigma\alpha&0&0\\
0&\beta+\sigma\alpha&0\\
0&0&-2
\end{pmatrix}.
 $$
The   constant matrix $\mathcal{M}$  has no eigenvalues with 
zero real part,   both the perturbation $\nu f(2\omega t) +\mu$ 
 and the non-linear part $R(X)$ are
 bounded and the non-linear part $R(X)$ is 
 uniformly Lipschitz in a compact neighbourhood of $\EU^2$. 
Under these conditions, Palmer's Theorem \cite[pp 754]{Palmer73} implies that there
 are neighbourhoods  of $\vv$ and $\ww$,  where the vector field is $C^1$ conjugate to its 
 linearisation.

\subsection{Local map near $\vv$}
\label{sec_v}
The calculation of the first return map will use the form of the general solution of \eqref{linV}.
For $z=1+w$ we get $\dot{w}= -2(1+w).$ 
Using the Lagrange method of variation of parameters as described in \cite[pp 842]{Shilnikov_book}, the solution of the linearised system \eqref{linV} near $\vv$,
with initial condition $(x,y,w)(s)=(1,y_1,w_1)\in \EU^1\times In(\vv)\backslash W^s(\vv)$ at time $s$,
 is:\\
\begin{equation}
\label{lin_solu_v}
\left\{ 
\begin{array}{l}
x(t,s)=\ee^{(\beta-\alpha)(t-s)}\left(1-\dpt\int_s^t \ee^{-(\beta-\alpha)(\tau-s)}  (\nu f(2\omega \tau) +\mu ) d\tau \right) \\ \\
y(t,s)= y_1 \ee^{(\alpha+\beta)(t-s)}\\ \\
w(t,s)= (w_1+1) \ee^{-2(t-s) }-1 .\\
\end{array}
\right.
\end{equation}
 
  The \emph{time of flight} $T_1$ from $\EU^1\times In(\vv)\backslash W^s(\vv)$ to $\EU^1\times Out(\vv)$
 is defined as the minimum non-negative time $T_1$ such a trajectory starting at $\EU^1\times In(\vv)\backslash W^s(\vv)$ hits $\EU^1\times Out(\vv)$.
 If $(1,y_1,w_1)\in \EU^1\times In(\vv)\backslash W^s(\vv)$ then $T_1(y_1,w_1)$  is the solution $T_1=t$ of $y(t)=1$.
 Hence:
 
 $$
y(T_1)=  1\quad \Leftrightarrow\quad y_1 \ee^{(\alpha+\beta)(T_1-s)} =1   \quad\Leftrightarrow \quad\ln\left(\frac{1 }{y_1}\right) = (\alpha+\beta)(T_1-s) .
$$
In this case, $T_1$ does not depend 
on $\nu$ or
on $\mu$. These solutions arrive at $Out(\vv)$ at a time $$T_1= s+\ln \left(\frac{1}{y_1}\right)^{\frac{1}{\alpha+\beta}} =s -\dfrac{1}{\alpha+\beta} \ln y_1.$$
 
Replacing $t$ by  $T_1$  in the first and third equations of (\ref{lin_solu_v}), we get
the map  
 $$
 {\Phi_{\vv}:\EU^1\times In(\vv)}
 \rightarrow \EU^1\times Out(\vv)
 $$
 \begin{equation}
\label{Phi_v1}
\Phi_{\vv}(s,y_1, w_1)  =
\left( \begin{array}{l}   s -\frac{1}{\alpha+ \beta} \ln {y_1}\\ \\
 y_1^{\delta  } 
\left(1-\dpt \int_s^{T_1} \ee^{-(\beta-\alpha)(\tau-s)}    (\nu f(2\omega \tau)+\mu ) d\tau \right)\\ \\
 (w_1+1){y_1}^{\frac{2}{\alpha+\beta}}-1 
\end{array}\right)
 =(T_1, \hat{x}_1, \hat{w}_1) .
\end{equation}

\subsection{Local map near $\ww$}
\label{sec_w}

The treatment of \eqref{linV} for $\sigma=-1$ is similar to Subsection ~\ref{sec_v}, although the computations involve more steps.
The solution of  \eqref{linV},
with initial condition $$(x,y,w)(s)=(x_2,1,w_2)\in \EU^1\times In(\ww)$$ at time $s$,  is:
\begin{equation}
\label{local_v-}
\left\{ 
\begin{array}{l}
x(t)= x_2 \ee^{(\alpha+\beta)(t-s)}\left(1-\dpt \frac{1}{x_2}\dpt \int_s^t \ee^{-(\alpha+\beta)(\tau-s)}  (\nu f(2\omega \tau) +\mu ) d\tau \right) \\ \\
y(t)=  \ee^{(\beta-\alpha)(t-s)}\\ \\
w(t)= (w_2-1) \ee^{-2(t-s) }+1 .
\end{array}
\right.
\end{equation}

 The time  $T_2$ of arrival at $Out(\ww)$, starting at $In(\ww)$
  is more difficult to compute than $T_1$.  This is why we use its  Taylor expansion at $(0,0)$ truncated at second order of $\nu$ and $\mu$.
We write $T_2(\nu, \mu)$ to stress its dependence on the bifurcation parameters.

 \begin{lemma} 
 \label{lemma1}
 The time of flight $T_2$ inside $V_-$ of  $(x_2,1,w_2)\in \EU^1\times In(\ww)$ only depends on $x_2$ and may be written as: 
 $$
 T_2(\nu,\mu)= s-\frac{1}{\alpha+\beta} \ln x_2+ \left[ 
\frac{1}{
(\alpha+\beta)}    \int_s^{T_2(0,0)} \ee^{-(\alpha+\beta)(\tau-s)}   (\nu f(2\omega \tau)+\mu ) d\tau \right] 
+ \mathcal{O}(||(\nu,\mu)||^2) ,
$$ 
where $\mathcal{O}(||(\nu,\mu)||^2)$ denotes the usual Landau notation
and $T_2(0,0)=s-\dfrac{\ln x_2}{\alpha+\beta}$. 
 \end{lemma}
 
\begin{proof}
Let us derive the Taylor expression of $T_2(\nu, \mu)$ of degree 1:
 
\begin{equation}
\label{Taylor1}
T_2(\nu, \mu)=T_2(0,0)+\frac{\partial T_2}{\partial \nu}(0,0)+\frac{\partial T_2}{\partial \mu}(0,0)+\mathcal{O}(||(\nu,\mu)||^2).
\end{equation}

By definition of time of flight in $V_-$, we may write:
\begin{eqnarray*}
1&\equiv& x(T_2(\nu,\mu),\nu,\mu)\\ & \overset{\eqref{local_v-}}=&\ee^{(\alpha+\beta)(T_2(\nu,\mu)-s)}\left(x_2-
 \nu\int_s^{T_2(\nu,\mu)}\ee^{-(\alpha+\beta)(\tau-s)}  f(2\omega \tau)  d\tau
 -\mu \int_s^{T_2(\nu,\mu)}\ee^{-(\alpha+\beta)(\tau-s)}  d\tau \right).
\end{eqnarray*}
For $\nu=\mu=0$, we get:
\begin{equation}\label{eq:T0}
 x_2 \ee^{(\alpha+\beta)(T_2(0,0)-s)}=1
 \quad\Rightarrow\quad
 T_2(0,0)=s-\dfrac{\ln x_2}{\alpha+\beta}.
\end{equation}
Using the Chain Rule applied to the equality $x(T_2(\nu,\mu),\nu,\mu)=1$ at  $(\nu, \mu)=(0,0)$ , we get
$$
\frac{\partial x}{\partial t}(T_2(0,0),0,0)\frac{\partial T_2}{\partial \nu}(0,0)+\frac{\partial x}{\partial \nu}(T_2(0,0),0,0)=0
$$
and thus:
$$
(\alpha+\beta)x_2 \ee^{(\alpha+\beta)(T_2(0,0)-s)}\frac{\partial T_2}{\partial \nu}(0,0)-
\int_s^{T_2(0,0)}\ee^{-(\alpha+\beta)(\tau-s)}  f(2\omega \tau)  d\tau=0.
$$
According to the left hand side of \eqref{eq:T0}, the previous equality may be simplified as:
\begin{eqnarray*}
&&(\alpha+\beta)\frac{\partial T_2}{\partial \nu}(0,0)-
\int_s^{T_2(0,0)}\ee^{-(\alpha+\beta)(\tau-s)}  f(2\omega \tau)  d\tau=0,\\
&\Leftrightarrow&
\frac{\partial T_2}{\partial \nu}(0,0)=
\frac{1}{(\alpha+\beta)}\int_s^{T_2(0,0)}\ee^{-(\alpha+\beta)(\tau-s)}  f(2\omega \tau)  d\tau.
\end{eqnarray*}
 Analogously, we have
$$
(\alpha+\beta)\frac{\partial T_2}{\partial \mu}(0,0)-
\int_s^{T_2(0,0)}\ee^{-(\alpha+\beta)(\tau-s)}  d\tau=0
$$
and hence
\begin{equation*}
\frac{\partial T_2}{\partial \mu}(0,0)=
\frac{1}{(\alpha+\beta)}\int_s^{T_2(0,0)}\ee^{-(\alpha+\beta)(\tau-s)}  d\tau.
\end{equation*}
Replacing  $T_2(0,0)$, $\dpt \frac{\partial T_2}{\partial \nu}(0,0)$ and $\dpt \frac{\partial T_2}{\partial \mu}(0,0)$ in   \eqref{Taylor1}, it yields:
 
$$
T_2(\nu,\mu)= s-\frac{1}{\alpha+\beta} \ln x_2+ \left[ 
\frac{1}{(\alpha+\beta)}    \int_s^{T_2(0,0)} \ee^{-(\alpha+\beta)(\tau-s)}   (\nu f(2\omega \tau)+\mu ) d\tau \right] 
+ \mathcal{O}(||(\nu,\mu)||^2) 
$$ 
and the result follows.
\end{proof}

From now on, we omit the remainder   $\mathcal{O}(||(\nu,\mu)||^2) $ of $T_2$ in the computations. Using   the expression of $T_2$ obtained in Lemma \ref{lemma1} in \eqref{local_v-},
 we may deduce that:

\begin{eqnarray*}
y(T_2)&=& e^{{(\beta-\alpha)}(T_2-s)}\\
&=& \exp \left(-\frac{(\beta-\alpha)}{\alpha+\beta} \ln x_2+ \left[ 
\frac{(\beta-\alpha)}{(\alpha+\beta)}    \int_s^{T_2(0,0)} \ee^{-(\alpha+\beta)(\tau-s)}   (\nu f(2\omega \tau)+\mu ) d\tau \right] \right) \\
&=& x_2^\delta\exp\left(-\delta   \int_s^{T_2(0,0)} \ee^{-(\alpha+\beta)(\tau-s)}   (\nu f(2\omega \tau)+\mu ) d\tau\right).\\
\end{eqnarray*}
Analogously, we may write:

\begin{eqnarray*}
w(T_2)&=&1+(w_2-1)\ee^{-2(T_2-s)}\\
&=& 1+(w_2-1)\exp \left( \frac{2}{\alpha+\beta} \ln x_2+ \left[ 
\frac{-2}{(\alpha+\beta)}    \int_s^{T_2(0,0)} \ee^{-(\alpha+\beta)(\tau-s)}   (\nu f(2\omega \tau)+\mu ) d\tau \right] \right)\\
&=& 1+(w_2-1) x_2^{2/(\alpha+\beta)}
 \exp \left(    
\frac{-2}{(\alpha+\beta)}    \int_s^{T_2(0,0)} \ee^{-(\alpha+\beta)(\tau-s)}   (\nu f(2\omega \tau)+\mu ) d\tau  \right).
 \end{eqnarray*}
Therefore, we 
define the local map as:
$$
{ \Phi_{\ww}:\EU^1\times In(\ww)}
\rightarrow \EU^1\times Out(\ww)
$$ 
 \begin{equation}
\label{Phi_w1}
\Phi_{\ww}=
\left(\begin{array}{l}
 s-\frac{1}{\alpha+\beta} \ln {x_2}+ \frac{1 }{
 (\alpha+\beta)}  
 \dpt  \int_s^{T_2(0,0)} \ee^{-(\alpha+\beta)(\tau-s)}   (\nu f(2\omega \tau) +\mu ) d\tau \\
  \\ 
x_2^\delta\exp\left(-\delta  
 \dpt  \int_s^{T_2(0,0)} \ee^{-(\alpha+\beta)(\tau-s)}  (\nu f(2\omega \tau) +\mu )d\tau \right) \\ 
 \\
1+(w_2-1)x_2^{2/(\alpha+\beta)}
\exp \left( 
{\frac{-2}{\alpha+\beta} \dpt \int_s^{T_2(0,0)} \ee^{-(\alpha+\beta)(\tau-s)}   (\nu f(2\omega \tau)+\mu )  d\tau }   \right) \dpt 
\end{array}\right)  
 =(T_2, \hat{y}_2, \hat{w}_2) .
\end{equation}

 It follows that if $x_2>0$ then $\hat{y}_2>0$. 
To establish a similar statement  for $\Phi_{\vv}$ we will need the information of Lemma~\ref{lem:K1} below.

\subsection{Summary}\label{subsecFirstReturn}
  The 
  expressions for   $\Phi_{\vv}$ (cf. \eqref{Phi_v1}) and $\Phi_{\ww}$ (cf. \eqref{Phi_w1})  may be written as:
%
$$
\Phi_{\vv}(s,y_1, { w_1}) =
\left(\begin{array}{l}
\dpt    s -\frac{1}{\alpha+ \beta} \ln y_1\\
{}\\
\dpt  y_1^{\delta  }(1-  K_1)
{}\\
\dpt {( w_1+1)}
y_1^{\frac{2}{\alpha+\beta}}
\end{array}\right)
=(T_1, \hat{x}_1, { \hat{w}_1)}
$$
and 
$$
\Phi_{\ww}(s,x_2, {w_2})=\left(
\begin{array}{l}
\dpt s-\frac{1}{\alpha+\beta} \ln x_2 + \frac{  K_2}{
(\alpha+\beta)}     \\
{}\\
\dpt  x_2^\delta\exp(-\delta K_2)\\
{}\\
\dpt  1+(w_2-1)\left(x_2^{2/(\alpha+\beta)}
\exp \left( 
\frac{-2K_2}{\alpha+\beta}      \right) \right)
 \end{array}\right)
 =(T_2, \hat{y}_2,{  \hat{w}_2) .}
$$
where 

 \begin{eqnarray*}
 \delta&=&  \frac{\alpha-\beta}{\alpha+\beta}>1\\
K_1&=& \int_s^{T_1} \ee^{-(\beta-\alpha)(\tau-s)}    (\nu f(2\omega \tau) +\mu) d\tau \\
 K_2&=&\int_s^{T_2(0,0)} \ee^{-(\alpha+\beta)(\tau-s)}   (\nu f(2\omega \tau) +\mu)  d\tau.
 \end{eqnarray*}
 
  Both
$K_1$ and $K_2$ depend on $s, \nu$ and $\mu$. Furthermore, when $\mu=\nu=0$, we get:
  \begin{eqnarray*}
 (T_1, \hat{x}_1, \hat{w}_1)&=& \left(s -\frac{1}{\alpha+ \beta} \ln y_1,  \,\, y_1^{\delta  }, \,\, (w_1+1)  y_1^{\frac{2}{\alpha+\beta}}\right)\\ \\
 (T_2, \hat{y}_2, \hat{w}_2) &=& \left(s-\frac{1}{\alpha+\beta} \ln x_2, \,\,  x_2^\delta, \,\, 1+(w_2-1) x_2^{2/(\alpha+\beta)}\right),
 \end{eqnarray*}
 corresponding to the expressions of the local maps for the unperturbed case.
  Note that, although the second and third coordinates of $\Phi_{\vv}$ are well defined and equal to zero at $(s,y_1,w_1)=(s,0,0)$, the first coordinate tends to infinity as $y_1$ goes to zero, since this point corresponds to the heteroclinic connection from $\vv$ to $\ww$. 
 In other words, $(s,0,0)\in In(\vv)$ is a point that never returns. The same remark applies to $\Phi_{\ww}$.

\subsection{Auxiliary result}
\label{ss:aux}
In order to have an explicit expression for $K_1$ and $K_2$, we assume that $f(t)=~\sin t$. The integrals $K_1$ and $K_2$ are linear on $\nu$ and $\mu$ and, under the previous assumption,  they may be computed explicitly. 
  \begin{lemma} \label{lem:K1}
The following equalities hold:
\begin{enumerate}
\item  
$$
\begin{array}{lcl}
K_1&=&\dfrac{\nu y_1^{-\delta}}{(\beta-\alpha)^2+4\omega^2}
\left[(\beta-\alpha)\sin(2\omega(T_1-s))+2\omega\cos(2\omega(T_1-s))\right]\\
&&\\
&&-\dfrac{\nu}{(\beta-\alpha)^2+4\omega^2}\left[(\beta-\alpha)\sin(2\omega s)+2\omega\cos(2\omega s)\right]{  +\mu\dfrac{y_1^{-\delta}-1}{\alpha-\beta}.}
\end{array}
$$

\item $$
\begin{array}{lcl}
K_2&=&\dfrac{\nu x_2}{(\alpha+\beta)^2+4\omega^2}
\left[(\alpha+\beta)\sin(2\omega(T_2(0,0)-s))+2\omega\cos(2\omega(T_2(0,0)-s))\right]\\
&&\\
&&-\dfrac{\nu}{(\alpha+\beta)^2+4\omega^2}\left[(\alpha+\beta)\sin(2\omega s)+2\omega\cos(2\omega s)\right]{+\mu\dfrac{1-x_2}{\alpha+\beta}.}
\end{array}
$$

\end{enumerate}
\end{lemma} 
\bigbreak
\begin{proof}
\begin{enumerate}
\item Using the linearity of the integral, $K_1$ may be written as  
$$
\nu\int_s^{T_1} \ee^{-(\beta-\alpha)(\tau-s)} f(2\omega \tau) d\tau+
\mu \int_s^{T_1} \ee^{-(\beta-\alpha)(\tau-s)} d\tau.
$$
Each summand may be computed explicitly. For the second one we have
$$
\int e^{-(\beta-\alpha)(\tau-s)} d\tau=\frac{-e^{-(\beta-\alpha)(\tau-s)} }{(\beta-\alpha)}.
$$
From \cite[Lemma 6]{LR18} (integrating by parts twice) and since $f(t)=\sin t$, we get
$$
 \int e^{-(\beta-\alpha)(\tau-s)}\sin(2\omega\tau) d\tau=
\dfrac{e^{-(\beta-\alpha)(\tau-s)}}{(\beta-\alpha)^2+4\omega^2}\left[
(\beta-\alpha)\sin(2\omega\tau) +2\omega \cos(2\omega\tau) 
\right]
$$
To compute $K_1$  we need  $T_1-s=-\dfrac{\ln y_1}{\beta+\alpha}$, hence
$$
e^{-(\beta-\alpha)(T_1-s)}=y_1^{(\beta-\alpha)/(\beta+\alpha)}=y_1^{-\delta}
$$
and the expression for $K_1$ follows. \\
\item Analogous computations.
\end{enumerate}
\end{proof}

 In particular it follows from this lemma that $K_1<1$ for sufficiently large $\omega$, since  in this case the terms with $\nu$ are small and the term with $\mu$ is negative.
Hence, for large $\omega$, if $y_1>0$ and $\Phi_{\vv}(s,y_1, w_1) =(T_1, \hat{x}_1,  \hat{w}_1)$ then $\hat{x}_1>0$.

\subsection{Global and first return maps}\label{ss: global}

The first return map to $ \EU^1\times In(\ww)$ will be 
\begin{equation}
\label{return1}
  \Rc_{(\nu,\mu)}:= (\Psi_{\vv\rightarrow \ww}) \circ\Phi_{\vv} \circ (\Psi_{\ww\rightarrow \vv}) \circ \Phi_{\ww}
 \end{equation}
  where
  $$
  \Psi_{\vv\rightarrow \ww}: \EU^1\times Out(\vv) \rightarrow \EU^1\times In(\ww)\quad  \text{and}\quad \Psi_{\ww\rightarrow \vv}: \EU^1\times Out(\ww) \rightarrow \EU^1\times In(\vv)
  $$
 are the global maps { whose expressions are given below.}
 
  Trajectories that remain close to the network $\Gamma$ spend long times near the equilibria 
 and make fast transitions from each  neighbourhood $V_\pm$ to the next one.
 Thus the  first components of $\Psi_{\vv\rightarrow \ww}$ and $ \Psi_{\ww\rightarrow \vv}$, representing the time transitions, may be disregarded.
  Since the symmetry $\kappa_2$ remains for  $\mu, \nu>0$, we may assume that 
 $\Psi_{\ww\rightarrow \vv}$ is the identity. Finally, since the cycle is broken for $\mu>0$, the map $\Psi_{\vv\rightarrow \ww}$ depends on $\mu$ in an affine way.
Therefore, for $a>0$, we take the transition maps as:
\begin{eqnarray*}
  \Psi_{\vv\rightarrow \ww}(s,\hat{x}_1, { \hat{w}_1}) \mapsto &  (s, \hat{x}_1 + a\mu,{ \hat{w}_1})&= (s, x_2, w_2)   \\
 \Psi_{\ww\rightarrow \vv}(s,\hat{x}_2, { \hat{w}_2}) \mapsto&  (s,\hat{x}_2, { \hat{w}_2})& =(s, x_1, w_1).\\
\end{eqnarray*}

\section{Main result}\label{sec:mainResult}

We are now in a position to show that when the frequency $\omega$ of the forcing tends to infinity,  the first return map for $F_{(\nu,\mu)}$  approaches that of the averaged system $F_{(0,\mu)}$.

\begin{theorem}\label{prop:main}
For initial conditions close to {any of the cycles in the network}
 $\Gamma$, the first return map  
$$
\Rc_{(\nu,\mu)}:{ \EU^1\times In(\ww)}
\rightarrow\EU^1\times In(\ww)
$$
for \eqref{general} satisfies
$$
\lim_{\omega\to\infty}\Rc_{(\nu,\mu)}=\Rc_{(0,\mu)}.
$$
\end{theorem}
\begin{proof}
To prove the result we do not need to write explicitly the analytical expression of $\Rc_{(\nu,\mu)}$. 
First of all, note that the global maps defined in Subsection \ref{ss: global} do not depend on $\nu$. From the expressions \eqref{Phi_v1} and \eqref{Phi_w1} derived in Subsection~\ref{subsecFirstReturn} it is clear that $\Phi_{\vv}$ and $\Phi_{\ww}$ only depend on $\omega$, $\nu$ and $\mu$ through the integrals $K_1$ and $K_2$.
Their expressions are computed explicitly in Lemma~\ref{lem:K1} above,  
where we show them to be of the form $$K_i(\omega, s, \nu, \mu)=\nu H_{i,\nu}(\omega,s)+\mu H_{i,\mu}(\omega,s),\quad i=1,2.$$
In both cases the term $H_{i,\nu}(\omega,s)$ contains one of the factors
$$
{1}/[(\beta\pm\alpha)^2+4\omega^2]
$$
multiplying a combination of sines and cossines, the last ones multiplied by $\omega$.
Hence, $H_{i,\nu}(\omega,s)$ consists of a term that goes to zero with $\omega^2$ when $\omega\to+\infty$ multiplying functions that are either bounded  or   the product of $\omega$ by 
something bounded.
Therefore 
$\dpt
\lim_{\omega\to\infty}H_{i,\nu}(\omega,s)=0
$
and the result follows.
\end{proof}

\section{Dynamics for $\omega$  large}
\label{sec:DinOmegaLarge}
We  have proved that the dynamics associated to $\dpt \lim_{\omega \rightarrow +\infty }F_{(\nu,\mu)}$ is qualitatively the same as that of the autonomous system $F_{( 0,\mu)}$. 
In this section we analyse the map $F_{( 0,\mu)}$ and obtain some dynamical information. Before going further, note that if $
\nu=0$ then:
$$
K_1= \mu\frac{y_1^{-\delta}-1}{\alpha-\beta}\qquad \text{and}\qquad K_2= \mu\frac{1-x_2}{\alpha+\beta}.
$$
Our first result is technical and provides an explicit expression for the components of $\mathcal{R}_{(0, \mu)}$. 

\begin{lemma} 
\label{return1,2,3}
For $\nu=0$ and $(s, x_2, w_2) \in  \EU^1\times In(\ww)$ {writing $\Rc_{(0,\mu)}=(h_1,h_2,h_3)$} the following equalities hold:\\
\begin{enumerate}
\item\label{item:h1}
$\dpt h_1(s, x_2, w_2)= s-\frac{1+\delta}{\alpha+\beta} \ln x_2 + \dfrac{K_2 (1+\delta)}{\alpha+\beta}$;\\
\item\label{item:h2}
 $\dpt h_2(s, x_2, w_2)=  
 { 
 x_2^{\delta^2}\exp\left(-\delta^2K_2\right)\left[1-\dfrac{\mu}{\alpha-\beta}\right]-\dfrac{\mu}{\alpha-\beta}+a\mu}$;\\
\item\label{item:h3}
 $\dpt h_3(s, x_2, w_2)=  \left[2+(w_2-1) x_2^{\frac{2}{\alpha+\beta}}\exp\left(\frac{-2K_2}{\alpha+\beta}\right)\right] \left[x_2^\delta \exp(-\delta K_2)\right]^{\frac{2}{\alpha+\beta}}$.\\ 
\end{enumerate}

\begin{proof}
The proof follows by composing of the local maps \eqref{Phi_v1} and \eqref{Phi_w1} derived in Subsection~\ref{subsecFirstReturn} and the global maps $  \Psi_{\vv\rightarrow \ww}$ and $  \Psi_{\ww\rightarrow \vv}$  in the order prescribed by $$ \Rc_{(0,\mu)}{=(\Psi_{\vv\rightarrow \ww}) \circ\Phi_{\vv} \circ (\Psi_{\ww\rightarrow \vv}) \circ \Phi_{\ww}}$$ cf. \eqref{return1}.
{ We start by obtaining for $\nu=0$ the expressions:
$$y_1=x_2^\delta\exp(-\delta K_2)\qquad
K_2=\mu\dfrac{1-x_2}{\alpha+\beta}\qquad
K_1=\mu\dfrac{y_1^{-\delta}-1}{\alpha-\beta}=\mu\dfrac{x_2^{-\delta^2}\exp(\delta^2 K_2)-1}{\alpha-\beta}.
$$
Hence, for \eqref{item:h2} we get:
\begin{eqnarray*} 
h_2(s, x_2, w_2)&=&y_1^{\delta  }(1-  K_1)\\
&=&  \left(x_2^\delta \exp(-\delta K_2)\right)^\delta \left(1-\mu \frac{(x_2^{-\delta}\exp(-\delta K_2))^{-\delta}-1}{\alpha-\beta}\right)+a\mu\\
 &=& x_2^{\delta^2} \exp(-\delta^2 K_2)- \dfrac{\mu}{\alpha-\beta}+\mu
 \dfrac{x_2^{\delta^2} \exp(-\delta^2 K_2)}{\alpha-\beta}+a\mu\\
 &=&  x_2^{\delta^2}\exp(-\delta^2K_2)\left[1-\dfrac{\mu}{\alpha-\beta} \right]
 -\frac{\mu}{\alpha-\beta}+a\mu .
\end{eqnarray*}
For \eqref{item:h1}:
\begin{eqnarray*} 
h_1(s, x_2, w_2)&=&s-\frac{1}{\alpha+\beta} \ln x_2 +\dfrac{K_2}{\alpha+\beta}-\frac{1}{\alpha+\beta}\ln y_1\\
&=& s-\frac{1+\delta}{\alpha+\beta} \ln x_2 + \frac{K_2 (1+\delta)}{\alpha+\beta} .
\end{eqnarray*}
Finally, for \eqref{item:h3}:
$$
w_1=1+(w_2-1)\left(x_2^{2/(\alpha+\beta)}\exp \left( \frac{-2K_2}{\alpha+\beta}      \right) \right)
$$
\begin{eqnarray*} 
h_3(s, x_2, w_2)&=&  {( w_1+1)}y_1^{\frac{2}{\alpha+\beta}}\\
&=& \left[2+(w_2-1) x_2^{\frac{2}{\alpha+\beta}}\exp\left(\frac{-2K_2}{\alpha+\beta}\right)\right] \left[x_2^\delta \exp(-\delta K_2)\right]^{\frac{2}{\alpha+\beta}} .
\end{eqnarray*}
}
\end{proof}

\end{lemma}
From now on, we are interested in the dynamics when $\nu=0$. 
The map $h_2$ just depends on $x_2$; this is why we may define $h_2:[0,1]\rightarrow \RR$ as
 $$h_2(x_2):= h_2(s, x_2, w_2).$$

\bigbreak
\begin{lemma}
\label{lemma_tec1}
The following assertions are valid for  $(s, x_2, w_2) \in  \EU^1\times In(\ww)$:\\
\begin{enumerate}
\item For $\mu=0$, $x_2=0$ is a hyperbolic attracting fixed point of $h_2(x_2)=x_2^{\delta^2}$; \\
\item 
If $\mu>0$, then:\\
\begin{enumerate}
\item\label{item:fixh2}
if $a-\frac{1}{\alpha-\beta}>0$, then $h_2$ has a hyperbolic attracting {fixed} point $x^\star$ of {order}
$\mathcal{O}(\mu)$;\\
\item\label{item:fixh3}
 if $\frac{4\alpha}{(\alpha+\beta)^2}>1$, then { for any $s\in\EU^1$ and any $x_2\in[0,1]$} the map {$w_2\mapsto h_3(s, x_2, w_2)$}
 is a Lipschitz contraction in the variable $w_2$.
\end{enumerate}
\end{enumerate}
\end{lemma}

\bigbreak
\begin{proof}
\begin{enumerate}
\item 
For $\mu=0$, we have $K_2=0$ and then $h_2(x_2)=x_2^{\delta^2}$ (cf. Item (2) of Lemma \ref{return1,2,3}), whose fixed point is $0$. Since $h_2'(0)=0$, the hyperbolicity and attractiveness follow. \\
\item 
\begin{enumerate}
\item
 For $\mu=0$ the graph of $h_2(x_2)$ crosses transversely the graph of the identity at $x_2=0$, then for small $\mu\ne 0$ the two graphs still cross transversely at a nearby value $x^*(\mu)$ of $x_2$, the hyperbolic continuation of the fixed point found in (1).
Since $h_2(0)= \mu \left(a-\frac{1}{\alpha-\beta}\right)$ then $x^*(\mu)>0$ if and only if $h_2(0)>0$ and in this case the fixed point is a hyperbolic attractor with $\dpt \lim_{\mu \rightarrow 0} x^\star(\mu)=0$.  The location of this fixed point is sketched in Figure \ref{fig:diagonal}.
 
\item
By item (3) of Lemma \ref{return1,2,3}, the map   $ h_3$ does not depend on $s$ and may be written as 
$$
h_3(s, x_2, w_2)= {C_1}(x_2)+{ C_2}(x_2)w_2
$$
where  ${ C_1},{C_2} : [0,1] \rightarrow \RR$ are smooth maps with
$$
\dpt {  C_2}(x_2)= x_2^\frac{2+2\delta}{\alpha+\beta} \exp\left( \frac{-2K_2(1+\delta)}{\alpha+\beta}\right) . 
$$
Since $$\frac{2+2\delta}{\alpha+\beta} = \frac{4\alpha}{(\alpha+\beta)^2}\overset{Hyp}>1$$   and $\left|\exp\left( \frac{-2K_2(1+\delta)}{\alpha+\beta}\right)  \right|<1$, then $h_3$ is a Lipschitz contraction in the variable $w_2$. 
\end{enumerate}
\end{enumerate}
\end{proof}
\begin{figure}
\begin{center}
\includegraphics[width=8.5cm]{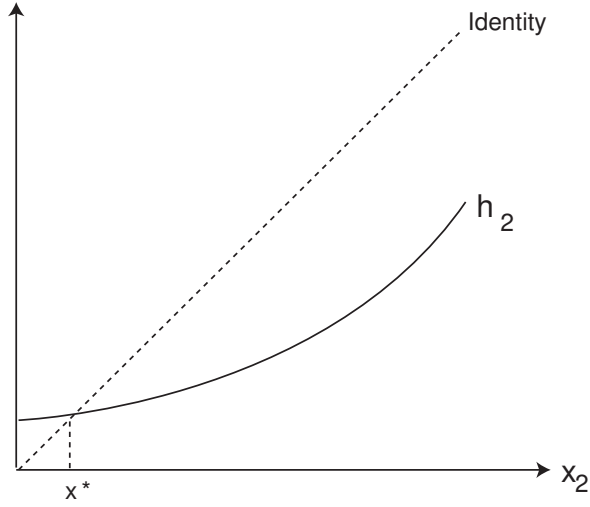}
 \end{center}
\caption{\small The graph of $h_2$ intersects the  Identity map once near the origin.   }
\label{fig:diagonal}
\end{figure}
 
Although $x_2=1$ is also a fixed point of $h_2$ (for $\mu=0$), we are disregarding this   point from the statement of Lemma \ref{lemma_tec1} by two reasons: first it is repelling and second it lies on the boundary of   the domain of $\mathcal{R}_{(0,0)}$.  
 Please note that the {restriction}
 $\frac{4\alpha}{(\alpha+\beta)^2}>1$ of 
{Lemma \ref{lemma_tec1}~\eqref{item:fixh3} is also used in  \cite[\S 4.6]{LR18} } 
for the reduced map.

\begin{theorem}\label{th:orbitaPer}
For   
$\nu=0$ and small $\mu>0$, 
{if $\dfrac{4\alpha}{(\alpha+\beta)^2}>1$ and $a>\dfrac{1}{\alpha-\beta}$, then}
the flow associated to $F_{(0,\mu)}$ has an attracting periodic solution whose period $P$ satisfies $P=\mathcal{O}(-\ln \mu )$. 
\end{theorem}
\begin{proof}
{Let $x^\star$ be}
the hyperbolic and attracting fixed point   of $h_2$ obtained in Lemma \ref{lemma_tec1}. 
Since the map {$w_2\mapsto h_3(s, x^\star, w_2)$}
 is a Lipschitz contraction in the variable $w_2$,  { it}
 has an attracting fixed point {$w^\star$.}
  that depends smoothly on $x_2$. 
  This means that there exists $P>0$ such that $$\mathcal{R}_{(0,\mu)}(s,x^\star,w^\star)=\mathcal{R}_{(0,\mu)}(s+P,x^\star,w^\star),$$ 
therefore  the flow associated to $F_{(0,\mu)}$ has an attracting periodic solution of period $P$.
 This periodic solution spends a time $(T_2(s,x^\star,w^\star)-s)$ inside $V_-$, followed,  inside $V_+$, by a flight time
$(T_1(s^+,y^+,w^+)-s^+)$ with $$(s^+,y^+,w^+)=\Psi_{\ww\rightarrow \vv}\circ \Phi_{\ww}(s,y^\star,w^\star).$$
Using Item (1) of Lemma \ref{return1,2,3}, this solution has period given by:
\begin{eqnarray*}
P&= & s-\frac{1+\delta}{\alpha+\beta} \ln x^\star + \frac{K_2 (1+\delta)}{\alpha+\beta}\\
&= & s-\frac{1+\delta}{\alpha+\beta} \ln x^\star + \mu \frac{1-x^\star}{\alpha+\beta}\frac{(1+\delta)}{\alpha+\beta}\\
 &\overset{x^\star=\mathcal{O}(\mu)}=&s-\frac{1+\delta}{\alpha+\beta} \ln x^\star + \mathcal{O}(\mu).
\end{eqnarray*}
Since $x^\star= \mathcal{O}(\mu)$ (by Lemma \ref{lemma_tec1}) then the period $P$ is of the order $\mathcal{O}(-\ln \mu)$ 
{completing the proof.}
 \end{proof}
From Theorems~\ref{prop:main} and \ref{th:orbitaPer} it follows immediately that:

\begin{corollary}
For small $\nu,\mu>0$ and very large $\omega$, if $\dpt \frac{4\alpha}{(\alpha+\beta)^2}>1$ and $a>\dfrac{1}{\alpha-\beta}$, then  the flow associated to 
 $F_{(\nu,\mu)}$ has an attracting periodic solution whose period $P$ 
  satisfies $P=\mathcal{O}(-\ln \mu )$.
 \end{corollary}

\section{Discussion and final remarks}
\label{ss: Discussion}
In this work, we investigate the influence of high frequency forcing on a differential equation exhibiting a clean  attracting heteroclinic network -- { \emph{clean} in the sense that the unstable manifolds of all the nodes lie in the network.}
Our result says that if the frequency of the non-autonomous perturbation goes to infinity, then the
 dynamics of the vector field is governed by  the averaged system:
the non-autonomous equation (\ref{general}) 
{behaves like an autonomous one.}
Our main result has been motivated by  Tsai and Dawes \cite{DT3, TD2, TD1} in the context of the Guckenheimer and Holmes example. They 
 claim without proof
  that the time-periodic forcing term has an effect equivalent to that of the time-averaged perturbation term.

 Our findings agree well with the theory developed by Cheng-Gui \emph{et al} \cite{Cheng}. They considered a system of the form 
\begin{equation}
\label{Gui}
\dot{x}= f(x)+B \cos (\omega t)H
\end{equation}
where { $x \in \RR^n$}
  represents the state vector of the nonlinear system, 
{$H=(1, 1, ..., 1)^T$, $f:\RR^n\rightarrow\RR^n$ is a nonlinear vector field}
and $B \cos(\omega t)$ denotes  a  forcing with frequency $\omega$ and amplitude $B$. The unforced system ($H=\overline{0}$) may exhibit stationary, periodic or chaotic behaviour for different system parameters.  They state that a general solution of \eqref{Gui}, say $x(t)$, may be written as the sum of  a slow motion
$X(t)$ and a fast motion $\Psi(t)$:
$$
x(t)=X(t) +\frac{1}{\omega} \Psi (t, \omega)
$$
where {$\Psi:\RR\times\RR^n\rightarrow\RR^n$}
is a $2\pi/\omega$ $t$--periodic function with zero mean. If $\omega \rightarrow +\infty$, then the solution is governed by the slow-motion which is the solution of the original unperturbed system. The effect of high-frequency forcing becomes apparent. 
Stating our result in their terms, as the forcing frequency tends to infinity the equation for the
{fast motion}
 drops out completely.

 In \S \ref{ss:aux},
we have used the map $f(t)=\sin t$ in \eqref{s: object} {and $H=~(1, 0, ..., 0)^T$} but our work is still valid for any smooth {$t$--periodic}
map $f$ with zero   average.

 Conditions under which \eqref{general} has a chaotic regime in the form of a suspended horseshoe map have been obtained in \cite{LR21}, where it appears through the destruction of an invariant torus.
We may conclude that  here when $\omega = +\infty$,  in the extended phase space (Equation \eqref{eq:suspended}) we cannot write the global map as in Section 4.3 of \cite{LR21}. 
More specifically, in the limit case $\omega = +\infty$, the set
 $$
 W^u(\EU^1\times \{\ww\})\cap In(\EU^1\times \{\vv\})
 $$ 
 is not a non-degenerate graph of a multimodal function.
   When $\omega = +\infty$ the necessary distortion to obtain chaos 
    does not hold so there is no
   guarantee of  the torus-breakdown effects.  This limit case was left open by Wang \cite[pp. 4391]{Wang}.

We conjecture that our result holds for any  attracting and clean heteroclinic network where the connections are one-dimensional.  Since
nonlinear systems driven by high frequency forcing are prevalent in nature and engineering, we expect that these results are valuable and helpful to those
 applications. 
 
\section*{Acknowledgements}

The authors are grateful to an anonymous referee, whose attentive reading and useful comments improved the final version of the article.

 \appendix
\section{Notation}
\label{s: Notation}
We list the main notation used in this paper in order of appearance with the reference of the section containing a definition.

\begin{table}[ht]
\begin{center}
\begin{tabular}{cll}
Notation & Meaning & Subsection    \\
&&\\
$\Gamma $ & Heteroclinic network formed by four cycles & \S \ref{ss:unp}\\
&&\\
$\delta$ & Saddle-value of $\vv$ and $\ww$ & \S \ref{ss:unp} \\
&&\\
$V_\pm$ & nieghbourhoods of $\vv$ and $\ww$ &  \S \ref{cross_sections}\\
&&\\
$In(\textbf{v}_\pm)$ & Cross-sections near $\textbf{v}_\pm$ for $F_{(0,0)}$&   \S \ref{cross_sections}\\
 $Out(\textbf{v}_\pm)$&&\\
 &&\\
$\EU^1\times In(\textbf{v}_\pm)$ & Augmented cross-sections for $F_{(\nu, \mu)}$ & \S \ref{secFirstReturn}\\
 $\EU^1\times Out(\textbf{v}_\pm)$&&\\
 &&\\
 $\Phi_{\vv}$ &  Local map from $ \EU^1\times In(\vv)$ to  $\EU^1\times Out(\vv)$ &  \S \ref{sec_v}  \\
&&\\
  $\Phi_{\ww}$ &  Local map from $ \EU^1\times In(\ww)$ to  $\EU^1\times Out(\ww)$ &  \S \ref{sec_w} \\
&&\\

$T_1$ & Time of flight  inside $V_+$    & \S \ref{sec_v}   \\
&&\\
$T_2$ & Time of flight  inside $V_-$    &\S \ref{sec_w}  \\
&&\\
$\Psi_{\vv\rightarrow \ww}$
 & Global map from $ \EU^1\times Out(\vv)$ to $\EU^1\times In(\ww)$ & \S \ref{ss: global}  \\
 &&\\
$\Psi_{\ww\rightarrow \vv}$
 & Global map from $ \EU^1\times Out(\ww)$ to $\EU^1\times In(\vv)$ & \S \ref{ss: global}\\
&&\\
$  \Rc_{(\nu,\mu)}$ & Return map to $\EU^1\times In(\ww)$ & \S \ref{ss: global} \\
&&\\

 $ (h_1,h_2,h_3)$ &Components of the map  $\Rc_{(0,\mu)}$ & \S \ref{sec:DinOmegaLarge}\\
&&\\
 $\dpt \frac{4\alpha}{(\alpha+\beta)^2}>1$  &Condition forcing $h_2$ to be the dominant in  $\Rc_{(0,\mu)}$ & \S \ref{sec:DinOmegaLarge}\\
&&\\
$x^\star$ & Stable fixed point of $h_2$ & \S \ref{sec:DinOmegaLarge}  \\
&&\\

\end{tabular}
\end{center}
\bigbreak
\caption{Notation} 
\label{summary}
\end{table}

\end{document}